\documentclass{article}
\usepackage[T1]{fontenc}
\usepackage[utf8]{inputenc}
\usepackage{lmodern}

\usepackage{amsmath, amssymb, amsthm, mathtools, bbm}
\usepackage{mathrsfs}

\usepackage[margin=1.15in]{geometry}
\usepackage{enumitem}
\usepackage{xcolor}
\usepackage{hyperref}
\usepackage[nameinlink,capitalize]{cleveref}

\hypersetup{
    colorlinks=true,
    linkcolor=blue!60!black,
    citecolor=blue!60!black,
    urlcolor=blue!60!black
}

\theoremstyle{plain}
\newtheorem{theorem}{Theorem}[section]
\newtheorem{proposition}[theorem]{Proposition}
\newtheorem{lemma}[theorem]{Lemma}
\newtheorem{corollary}[theorem]{Corollary}

\theoremstyle{definition}

\newtheorem{remark}[theorem]{Remark}

\crefname{theorem}{Theorem}{Theorems}
\crefname{proposition}{Proposition}{Propositions}
\crefname{lemma}{Lemma}{Lemmas}
\crefname{corollary}{Corollary}{Corollaries}
\crefname{conjecture}{Conjecture}{Conjectures}
\crefname{claim}{Claim}{Claims}
\crefname{definition}{Definition}{Definitions}
\crefname{example}{Example}{Examples}
\crefname{notation}{Notation}{Notations}
\crefname{question}{Question}{Questions}
\crefname{problem}{Problem}{Problems}
\crefname{remark}{Remark}{Remarks}
\crefname{equation}{Equation}{Equations}
\crefname{section}{Section}{Sections}

\newcommand{\EE}{\mathbb{E}}

\newcommand{\NN}{\mathbb{N}}

\newcommand{\PP}{\mathbb{P}}

\newcommand{\RR}{\mathbb{R}}

\newcommand{\ZZ}{\mathbb{Z}}

\newcommand{\cD}{\mathcal{D}}
\newcommand{\cE}{\mathcal{E}}

\newcommand{\cM}{\mathcal{M}}
\newcommand{\cN}{\mathcal{N}}

\newcommand{\cP}{\mathcal{P}}

\newcommand{\cS}{\mathcal{S}}
\newcommand{\cT}{\mathcal{T}}

\newcommand{\cX}{\mathcal{X}}

\newcommand{\Cov}{\mathrm{Cov}}
\newcommand{\dCov}{\mathrm{Dvar}}
\newcommand{\Cor}{\mathrm{Cor}}
\newcommand{\sdiam}{\mathrm{sdiam}}
\newcommand{\ddiag}{\mathrm{ddiag}}
\newcommand{\Var}{\mathrm{Var}}

\newcommand{\brac}[1]{\left({#1}\right)}
\newcommand{\bracc}[1]{\left[{#1}\right]}

\DeclarePairedDelimiter{\abs}{\lvert}{\rvert}

\newcommand{\eps}{\varepsilon}

\newcommand{\op}{\mathrm{op}}

\newcommand{\one}{\mathbf{1}}

\definecolor{AfonsoBlue}{RGB}{30,65,123}

\title{
Mixing of Glauber Dynamics on High Overlap Gibbs Measures
}

\author{
  Afonso S. Bandeira\thanks{ASB: Department of Mathematics, ETH Zürich, Rämistrasse 101, 8092 Zurich, Switzerland. \texttt{bandeira@math.ethz.ch}} \and 
  Ahmed El Alaoui\thanks{AE: Department of Statistics and Data Science, Cornell University, Ithaca, New York, USA. \texttt{elalaoui@cornell.edu}} \and
  Almut Rödder
\thanks{AR: Department of Mathematics, ETH Zürich, Rämistrasse 101, 8092 Zurich, Switzerland. \texttt{almutmagdalena.roedder@ifor.math.ethz.ch}}
}

\date{\today}

\begin{document}

\maketitle

\begin{abstract}

We show fast mixing of Glauber dynamics for certain quadratic Gibbs measures with large external fields.    
The main ingredient is an overlap condition that allows us to control correlation matrices uniformly over all pinnings, by controlling norms of small submatrices of the interaction matrix. Using stochastic localization, we then obtain a lower bound on the spectral gap and, consequently, polynomial-time mixing of Glauber dynamics.

As a direct application, we consider the Sherrington–Kirkpatrick model, whose interaction matrix is a scaled GOE matrix. For this model, we show that for any fixed finite inverse temperature~$\beta$, there exists a strength of external field $\theta$, not depending on the size of the system, for which Glauber dynamics mixes in polynomial time (with high probability on the draw of the interaction matrix). 
\end{abstract}

\setcounter{tocdepth}{2}
\tableofcontents

\section{Introduction}

The main object of study of this paper are quadratic Gibbs measures over the binary hypercube $\{\pm 1\}^n$
\begin{equation}\label{eq:quadraticGibbsMeasure1}
\nu(x) \propto \exp\brac{\frac{1}{2}x^\top J x + h^\top x}
\end{equation}
where $J \in \RR^{n \times n}$ is a symmetric interaction matrix and $h\in \RR^n$ an external field. Probability distributions of the form~\eqref{eq:quadraticGibbsMeasure1} are a central object of study in several fields, including Statistical Physics where different choices of $J$ and $h$ capture models, such as Curie-Weiss, Ising, Sherrington-Kirkpatrick, Edwards-Anderson, among others. In Statistics, such distributions arise as posterior distributions of various statistical models such as $\ZZ_2$ Synchronization and Community Detection in the Stochastic Block Model.

One of the most important questions surrounding these probability distributions is whether they can be sampled efficiently, meaning in time polynomial in $n$. A particularly important sampling algorithm is Glauber Dynamics, a Markov Chain Monte Carlo algorithm where, at each iteration, a random spin gets picked and its spin redrawn at random with respect to its conditional distribution (also referred to as a heat bath on one coordinate).

There is a rich line of work showing fast mixing of Glauber Dynamics under various conditions on $J$ (e.g. (\cite{eldan2022spectral,chen2022localizationschemesframeworkproving,anari24trickle}). Particularly remarkable is the spectral radius condition of Eldan, Koehler, and Zeitouni which guarantees fast mixing of Glauber from any starting point as long as the spectral diameter of $J$ is smaller than $1$, for any external field $h$.

The model that serves as the main motivation for the results of this paper is the celebrated Sherrington-Kirkpatrick model, where the interaction matrix has random Gaussian entries. More precisely, the Sherrington-Kirkpatrick Gibbs measure at inverse temperature $\beta$ and with external field $\theta\one$ is the probability measure
        \[
    \nu(x) \propto \exp\left(\frac{\beta}2x^\top W x + \theta \one^\top x\right),
        \]
    where $W\sim\mathrm{GOE}$ (symmetric, $W_{ij}\sim\mathcal{N}(0,1/n)$ for $i\neq j$ and $W_{ii}\sim\mathcal{N}(0,2/n)$).

The spectral diameter condition of~\cite{eldan2022spectral}, together with the observation that $\|\beta W\|\sim 2 \beta$ with high probability, implies fast mixing of cold-started Glauber dynamics for $\beta<\frac14$ (with high probability on the randomness of the interaction matrix). This has since been improved to $\beta < 0.295$ in~\cite{anari24trickle}. It is conjectured that fast mixing may hold for any $\beta<1$ (see Open Problem 15 in~\cite{randomstrasse101-2024}). 
Beyond Glauber dynamics sampling algorithms based on stochastic localization \cite{el2022sampling} are known to converge to the target distribution in Wasserstein distance \cite{el2022sampling, celentano2024sudakov} in the full regime $\beta <1$, and a recent result suggests a sampling algorithm that converges in total variation distance for $\beta < 0.5$~\cite{davies2026potential}.
However, in the setting where there is no external field, it is expected that mixing is super-polynomially slow for $\beta >1$, where the model is expected to enter a replica symmetry breaking region~\cite{de1978stability,lopatto2026replica}.
Recent progress in this direction includes the stable-algorithm obstruction~\cite{el2022sampling} in the full regime $\beta >1$ and exponentially slow worst-case Glauber mixing for sufficiently large $\beta$ ~\cite{sellke2025exponentially}.

On the other hand, replica symmetry can hold beyond $\beta >1$, provided that the external field strength $\theta$ is sufficiently large. In particular, for every $\beta \geq 0$ there exists $\theta \in \RR$ that places the model in the replica symmetric phase (the smallest such $\theta$ is known as the Almeida-Thouless (AT) line, for which replica symmetry has been predicted in~\cite{de1978stability} and recently shown in~\cite{lopatto2026replica}). While approximate message passing algorithms have been shown to quickly converge to so-called TAP solutions~\cite{bolthausen2014iterative} in this regime, polynomial time mixing of local Markov chains like Glauber dynamics in the low-temperature regime below the AT line remained a long-standing open question.

In this work, we address this question and identify a region that is covered by the predictions of~\cite{de1978stability} where fast mixing of Glauber dynamics can be shown.
While our methods do not appear to be strong enough to establish fast mixing of Glauber dynamics up to the AT line, they do establish the qualitative fact: For every inverse temperature $\beta\geq 0$, there exists a strength of the external field $\theta \geq 0$ (not depending on $n$) for which SK at that temperature and external field enjoys fast Glauber mixing from any initialization. Establishing this all the way to the AT line remains an open problem.

Most existing results showed fast mixing regardless of the external field. In the setting of the SK model at $\beta >1$ and large $\theta$ such a statement cannot hold: otherwise it would be possible to take an external field that cancels out the existing one and moves the model to the replica symmetry breaking side of the AT-line. In contrast, our arguments explicitly exploit the large external field. Such an external field forces draws of the model to be highly magnetized and we leverage this to control the correlation matrices arising in the spectral independence~\cite{anari2021spectral}, or stochastic localization with pinnings~\cite{chen2022localizationschemesframeworkproving}, framework. 

Our tools can be applied in settings beyond the SK model, including in showing fast mixing of Glauber dynamics for the posterior distribution arising in $\ZZ_2$ synchronization with side information.

\textbf{Organization of the Paper:}
The rest of the paper is organized as follows. In Section~\ref{secsub:mainresults} we state our main results (Theorem~\ref{thm:smallsubmatricesimplymixing} and Corollary~\ref{cor:appliedtoSK}). In Section~\ref{secsub:proofstrategy} we briefly describe our arguments and proof strategy. We include a notation and definitions section (\ref{secsub:notationdefinitions}).
In Section~\ref{sec:highoverlap} we derive the correlation bounds based on a conditional high overlap guarantee for replicas (Lemma~\ref{lemma:highoverlapreplicas}) and verify this condition for large external fields through a spin alignment guarantee (Corollary~\ref{corollary:manyquantifiers}).
In Section~\ref{sec:mainproofs}, we apply the previous results to all pinned submeasures to derive a lower bound on the spectral gap based on the stochastic localization approach and conclude the proof of the main Theorem~\ref{thm:smallsubmatricesimplymixing}. We prove Corollary~\ref{cor:appliedtoSK} in the same Section. 

We include a short discussion on further applications and open questions in Section~\ref{sec:otherapplicationsandquestions}.

\subsection{Main Results}\label{secsub:mainresults}
Our main result is Theorem~\ref{thm:smallsubmatricesimplymixing}, which states that if all principal submatrices (of the interaction matrix) of a certain size have small operator norm, there exists a threshold on the size of the external field above which Glauber dynamics mixes in polynomial time. 
\begin{theorem}\label{thm:smallsubmatricesimplymixing}
    Let $\delta\in(0,1)$ and let $J\in\RR^{n\times n}$ symmetric such that
    \[
    \max_{\substack{\cS\subset [n] \\ |\cS|\leq \delta n}}\left\|J_{\cS,\cS}\right\|\leq \frac15,
    \]
    then, there exists $\theta\geq 0$, which may depend on $\|J\|$ and $\delta$ but not on $n$, 
    such that lazy Glauber Dynamics for the quadratic Gibbs distribution on $\{\pm 1\}^n$ given by
    \[
    \nu(x) \propto \exp\left(\frac12x^\top J x + \theta \one^\top x\right),
    \]
    mixes, in total variation distance, in time polynomial on $n$. Moreover $\theta$ can be taken to be a universal constant times $\|J\|/{\delta}$.
\end{theorem}

A key feature of this result is that the required field strength is independent of $n$ when $\|J\|$ and $\delta$ can be guaranteed to be constant.

We apply this result to the Sherrington-Kirkpatrick model in Corollary~\ref{cor:appliedtoSK}. In this setting, the interaction matrix is a scaled GOE matrix which satisfies the assumptions with high probability. The size of the external field, for which polynomial time mixing can be guaranteed, depends on the inverse temperature $\beta$, but not on $n$, the size of the system.

\begin{corollary}\label{cor:appliedtoSK}
    For any inverse temperature $\beta$, and $n$ large enough, there exists $\theta\geq 0$, which may depend on $\beta$ but not on $n$, such that lazy Glauber Dynamics for the Sherrington-Kirkpatrick Gibbs measure at inverse temperature $\beta$ and with external field $\theta\one$ 
        \[
    \nu(x) \propto \exp\left(\frac{\beta}2x^\top W x + \theta \one^\top x\right),
        \]
    where $W\sim\mathrm{GOE}$ (symmetric, $W_{ij}\sim\mathcal{N}(0,1/n)$ for $i\neq j$ and $W_{ii}\sim\mathcal{N}(0,2/n)$), mixes, in total variation distance, in time polynomial in $n$ with high probability.
\end{corollary}

\subsection{Proof Strategy}\label{secsub:proofstrategy}
We show an upper bound on the (cold-started) mixing time by lower bounding the spectral gap of Glauber dynamics. We establish such a lower bound within the stochastic localization framework on pinned measures (Lemma~\ref{lemma:spectralindconditionsformixing}) by controlling all pinned measures.
More precisely, the proof proceeds by splitting pinned measures in two regimes: If at most $\delta n$ spins remain unpinned, the small-submatrix assumption guarantees that the pinned measure is effectively of high temperature. By the stochastic localization approach, it suffices to uniformly bound the correlation matrices of all pinned measures that still have many non-pinned spins.

The difficulty in controlling correlation matrices for such pinned measures is that the interaction matrix of the pinned measure is large and does not satisfy the small spectral radius condition. Instead, we use a reduction on high-overlap replicas, Lemma~\ref{lemma:highoverlapreplicas}: If two independent replicas have high conditional overlap, the spectral norm of the correlation matrix can be controlled by small-subset covariance matrices (corresponding to the small set of spins in which the replicas disagree). 

We show that a large external field forces many spins to align with the external field: 
Lemma~\ref{lemma:manyquantifiers} (more precisely, its one-spin conditioned version, Corollary~\ref{corollary:manyquantifiers}) states that for any desired constant overlap, one can pick $\theta$ large enough as to guarantee that, with high probability, two replicas have at least that level of overlap, resulting in suitably small disagreement sets where the principal submatrix condition can be deployed.

Since under a pinning, the original external field is perturbed, we need to consider the worst-case pinning-induced external field. Proposition~\ref{proposition:allpinningsgiveboundedexternalfields} states that the perturbation of the external field is only of order $O(\sqrt{n})$.
In this case, Lemma~\ref{lemma:manyquantifiers} still guarantees strong alignment with the original external field, in any pinned model with sufficiently many unpinned spins.

A uniform bound on the correlation matrices of all pinned measures up to a pinning size of $(1-\delta)n$ then follows from combining Lemma~\ref{lemma:highoverlapreplicas} and~\ref{proposition:allpinningsgiveboundedexternalfields} with Lemma~\ref{corollary:manyquantifiers}.
An application of Lemma~\ref{lemma:spectralindconditionsformixing}, completes the proof of a polynomial time upper bound on the mixing time for Glauber dynamics. Corollary~\ref{cor:appliedtoSK} follows by verifying the small-submatrix condition for GOE matrices by a union bound.

\subsection{Notation and Definitions}\label{secsub:notationdefinitions}

The measure $\nu$ is defined on the binary hypercube $\{\pm 1\}^n$ for $n \in \NN$. 
However, we will also need to define Gibbs measures on $\{\pm 1\}^m$, $m \leq n$, and for any external field  $h \in \RR^m$ as a framework to include all pinnings: we denote by $\nu_{J,h}(x)$ the quadratic Gibbs measure on the hypercube $\{\pm 1\}^m$
    \begin{align}\label{eq:Gibbs}
        \nu_{J,h}(x) \propto \exp\brac{\frac{1}{2} x^\top J x + h^\top x},
    \end{align}
    for a symmetric interaction matrix $J\in \RR^m$ and an external field vector $h \in \RR^m$.
    When the external field is zero we simply write $\nu_J$.

For a matrix $J \in \RR^{n \times n}$, we use  $J^{(\cT)} =J_{\cT,\cT}$ for $\cT \subseteq [n]$ to denote a principal submatrix, and use $\sdiam(J)=\lambda_{\max}(J)-\lambda_{\min}(J)$ to denote its spectral diameter.

    Given a random variable $X$ on the hypercube where each coordinate has non-zero variance we define
    $$\Cov(X) = \EE(X-\EE(X))(X-\EE(X))^\top$$
    $$\dCov(X) = \ddiag(\Cov(X)),$$
    where $\ddiag(M)$ is the diagonal matrix where the diagonal entries match the ones from $M$. The correlation matrix is given by
    $$\Cor(X) = \dCov(X)^{-\frac12}\Cov(X)\dCov(X)^{-\frac12}.$$
For a measure $\nu$, we denote by $\Cov_\nu$, $\Cor_\nu$ and $\dCov_\nu$ the respective matrices for $X\sim \nu$.

     For two independent copies $X,X' \sim \nu$ on $\{\pm 1\}^m$ define the (unnormalized) overlap to be
\begin{align}\label{eq:def:overlap}
    \rho(X,X') =\sum_{i=1}^m \mathbbm{1}_{X_i = X_i'}.
\end{align}

\section{High Overlap Distributions}\label{sec:highoverlap}

In this section we study distributions with high overlap, meaning that two independent replicas have high overlap with high probability. We start by showing that the correlation matrix of such a distribution can be controlled by covariances on small disagreement set, and then show that the distributions we study in this paper are indeed high overlap distributions.

\subsection{Correlation Bounds based on High Overlap}

We start by characterizing the law of the difference between two independent replicas, in their disagreement set.
We will see that, conditioned on the disagreement set of two replicas, the difference between them is distributed according to a new quadratic Gibbs measure, on a smaller support. 

\begin{lemma}\label{lem:newGibbs}
 Let $X,X'$ be iid draws from the quadratic Gibbs measure $\nu_{J, h}$ (as defined in~\eqref{eq:Gibbs}). Define $Z=X-X'$, the disagreement set $\cD \coloneqq \{i: X_i \neq X_i'\}$, and $Z_\cD$ as the restriction of $Z$ to the coordinates in $\cD$.
 It holds that the random variable $\frac12 Z_\cD \mid \cD$ is distributed according to a Gibbs measure of the form~\eqref{eq:Gibbs} with interaction matrix $2J_{\cD, \cD}$ and zero external field.
\end{lemma}
\begin{proof}
    Let $X,X' \sim \nu_{J,h}$. For every $i \notin \cD$, $Z_i=0$ by definition.
    Therefore, the measure $\PP_{X,X' \sim \nu}(Z_\cD \mid \cD)$ is supported on $\{\pm 2\}^{\cD}$ and  $\tilde{Z} = \frac12 Z_\cD$ on $\{\pm 1\}^{\cD}$. 
    For every $z \in \{\pm 1\}^{\cD}$ it holds
    
    \begin{align*}
        \PP_{X,X' \sim \nu}(\tilde{Z } = z \mid \cD)
        &= \frac{1}{\PP(\cD)}\sum_{y \in \{\pm 1\}^{\cD^c}} \PP_{X,X' \sim \nu}(X_\cD = z, X'_\cD = -z, X_{\cD^c} = X'_{\cD^c} = y)\\
        &= \frac{1}{\PP(\cD)}\sum_{y \in \{\pm 1\}^{\cD^c}}  \nu(X_\cD = z,  X_{\cD^c} = y) \nu( X'_\cD = -z, X'_{\cD^c} = y)\\
        &\propto \sum_{y \in \{\pm 1\}^{\cD^c}} \exp(z^\top J_{\cD,\cD} z + y^\top J_{\cD^c, \cD^c} y +2 h_{\cD^c}^\top y)\\
        &\propto \exp\left(\frac{1}{2}z^\top 2J_{\cD, \cD} z \right),
    \end{align*}
    where the normalizing constant does not depend on $z$. In the third step we used that the interactions and the external field cancel since  $(y+y)^\top J_{\cD, \cD^c}(z-z) + (z-z)^\top h_{\cD} =0$.
    Therefore, the law of $\frac12 Z_\cD \mid \cD$ is defined by a Gibbs measure of the form~\eqref{eq:Gibbs} with interaction matrix $2J_{\cD, \cD}$ and zero external field.
\end{proof}

We are now ready to show that if a distribution has high overlap, one can bound the norm of the correlation matrix by covariances on small subsets.

In order to bound the correlation matrices such a high overlap guarantee still needs to hold after conditioning on the disagreement of a single spin. This guarantees that the influence of a single event is marginal on the replica overlap.

\begin{lemma}\label{lemma:highoverlapreplicas}
Let $\nu_{J,h}$ be a quadratic Gibbs probability measure over $\{\pm 1\}^m$ and let $\delta \in (0,1)$. Let $X,X'$ be two independent copies drawn from $\nu_{J,h}$. If there exists $\eps_m$ such that, for any $i \in [m]$,
    \begin{align}\label{eq:overlap}
         \PP_{X,X' \sim \nu}\Big(\rho(X,X') \leq (1-\delta)m \Bigm| X_i \neq X_i'\Big) \leq \frac{\eps_m}{m},
    \end{align}
    Then,
    \begin{align*}
        \| \Cor(X) \|_\op \leq \sup_{\cS \subseteq [m]: |\cS| \leq \delta m }  \|\Cov(\nu_{2 J_{\cS,\cS}})\|_\op + \eps_m.
    \end{align*}
    Here $\rho(X,X')$ is the overlap as defined in~\eqref{eq:def:overlap}, $J_{\cS,\cS}$ is the principal submatrix of $J$ on the index set $\cS$, and $\nu_{2 J_{\cS,\cS}}$ is the quadratic Gibbs measure on $\{\pm1\}^{\cS}$ with quadratic disorder $2J_{\cS,\cS}$ and zero external field.
\end{lemma}

\begin{proof}
    For any vector $a \in \RR^m$ we have
    \begin{align*}
        a^\top \Cov(X) a  &= \Var \brac{a^\top X}= \frac12 \EE \bracc{\brac{\sum_{i=1}^m a_i (X_i-X_i')}^2}\\
        &= \frac12 \EE \bracc{ \EE \bracc{\brac{\sum_{i=1}^m a_i (X_i-X_i')}^2 \Biggm| \cD}},\\
    \end{align*}
    where $\cD \coloneqq \{i: X_i \neq X_i'\}$ denotes the disagreement set; the equality in first row holds since for any iid $Y$ and $Y'$, $\EE\bracc{(Y-Y')^2} = 2 \EE\bracc{Y^2} - 2 \EE[Y]\EE[Y'] = 2 \Var(Y)$; the last equality holds by the tower property of expectation.

Let $Z=\frac{X-X'}{2}$. A key observation is that the law of $Z \mid \cD$ is given by: $Z_{\cD^c}=0$ and $Z_{\cD}$ is a symmetric Gibbs measure of form~\eqref{eq:Gibbs}, supported on $\{\pm 1\}^{\cD}$, with interaction matrix given by $2 J_{\cD, \cD}$ and zero external field (see Lemma~\ref{lem:newGibbs} for a proof of this fact; this strategy dates back, at least, to~\cite{percus1975correlation}).

If $|\cD| \leq \delta m$ we have
\begin{align*}
    \EE \bracc{\brac{\sum_{i=1}^m a_i (X_i-X_i')}^2 \Biggm| \cD} &= \Var\brac{2 a_{\cD}^\top Z_{\cD} \mid \cD} \\
    &= 4 a_{\cD}^\top \Cov(\nu_{2 J_{\cD,\cD}})a_{\cD} \leq   4 \|a_{\cD}\|^2 \|\Cov(\nu_{2 J_{\cD,\cD}})\|,
\end{align*}
where $a_{\cD}$ is the vector $a$ restricted to the coordinates in $\cD$.

Therefore, taking expectation over $\cD$, we have
\begin{align*}
    &\EE \bracc{ \EE \bracc{\brac{\sum_{i=1}^m a_i (X_i-X_i')}^2 \Biggm| \cD} \mathbbm{1}_{\{|\cD| \leq \delta m\}}}
    \leq
    \EE \bracc{ 4 \|a_{\cD}\|^2 \|\Cov(\nu_{2 J_{\cD,\cD}})\| \mathbbm{1}_{\{|\cD| \leq \delta m\}}}  \\
        &\leq
    4\EE \bracc{  \|a_{\cD}\|^2 } \sup_{\cS \subseteq [m]: |\cS| \leq \delta m} \|\Cov(\nu_{2 J_{\cS,\cS}})\|  
    \leq 4 \Bigg(\sum_{i=1}^m a_i^2 \, \PP\brac{i \in \cD}\Bigg)\sup_{\cS \subseteq [m]: |\cS| \leq \delta m} \|\Cov(\nu_{2 J_{\cS,\cS}})\| \\
    &= 2 \left(a^\top \dCov(X) a\right)   \sup_{\cS \subseteq [m]: |\cS| \leq \delta m} \|\Cov(\nu_{2 J_{\cS,\cS}})\|,
\end{align*}
the last equation holds since $\PP(X_i \neq X_i') =\frac12\Var(X_i)$.

It remains to consider $|\cD| > \delta m$, i.e. $\rho(X,X') < (1-\delta) m$ which will hold with low probability. In this case we use the trivial bound that the covariance matrix of a distribution on $\{\pm2,0\}^m$ has spectral norm at most $4m$: 
\begin{align*}
    \EE \bracc{\brac{\sum_{i=1}^m a_i (X_i-X_i')}^2 \Biggm| \cD} \leq \|a_{\cD}\|^2 4 m.
\end{align*}
Moreover, 
\begin{align*}
         &\EE \bracc{ 4m\,\mathbbm{1}_{\{|\cD| > \delta m \}} \|a_{\cD}\|^2 }
         = \sum_{i=1}^m\EE \bracc{ 4m\,\mathbbm{1}_{\{|\cD| > \delta m \}} \mathbbm{1}_{i\in\cD} \, a_i^2 }\\
         &\leq 4m \sum_{i=1}^m a_i^2 \PP\brac{i \in \cD} \PP\brac{|\cD| > \delta m \mid i \in \cD}
         \leq 4  \eps_m \sum_{i=1}^m a_i^2 \PP\brac{i \in \cD} = 2  \eps_m a^\top \dCov(X) a .
\end{align*}

We are now ready to finish the proof. Let $b\in\RR^m$ such that $\|b\|=1$ and $a = \dCov(X)^{-\frac12} b$ which is well-defined for quadratic Gibbs measures. Note that $\|b\|^2=a^\top \dCov(X) a$. We have    \begin{align}\label{eq:finalresult}
        b^\top \Cor(X) b &= a^\top \Cov(X)a \nonumber\\
         &\leq \frac12 \brac{2 \|b\|_2^2 \sup_{\cS \subseteq [m]: |\cS| \leq \delta m} \|\Cov(\nu_{2 J_{\cS,\cS}})\| +2\eps_m\|b\|_2^2 }\nonumber\\
         &=   \sup_{\cS \subseteq [m]: |\cS| \leq \delta m} \|\Cov(\nu_{2 J_{\cS,\cS}})\| + \eps_m
    \end{align}
\end{proof}

\subsection{Spin Alignment in Quadratic Gibbs Measures}
In this section we prove that a large constant external field $h=\theta \one$ forces alignment of many spins with it (even after adding a smaller external field $v$, potentially pointing in a different direction). 
It is important to allow the adding of an arbitrary external field $v$ as one needs to control pinned submeasures of $\nu$ for any worst-case pinning (and a pinning induces an external field on the un-pinned spins).

\begin{lemma}\label{lemma:manyquantifiers}
For all $C_1,C_2\geq 0$ and $\delta \in (0,1)$, there exists $\theta,c_3,m_0\geq 0$ such that the following holds: for all $m \geq m_0$, $J \in \RR^{m \times m}$ symmetric, $v \in \RR^m$, if $\|J\|\leq C_1$ and $\|v\|\leq C_2\sqrt{m}$, the random variable $X$ with quadratic Gibbs distribution on $\{\pm 1\}^m$ given by
    \begin{equation}\label{eq:gibbsinlemmamanyquantifiers}
    \nu_{J,\theta \one_m +v}(x) \propto \exp\left(\frac12x^\top J x + \theta \one^\top x + v^\top x  \right),
    \end{equation}
    has, with probability at least $ 1-\exp(-c_3m)$, more than $(1-\frac{\delta}2)m$ positive spins.
\end{lemma}

\begin{proof}
    Throughout the argument, we will take $\theta$ satisfying:  \begin{align}\label{eq:bound_theta}
        \theta > 4 \max \left\{\frac{ C_2}{\sqrt{\delta}}, \brac{1+ \frac{2}{\sqrt{\delta}}}C_1,\log\brac{\frac{5e}{\delta}}\right\}.
    \end{align}
    
    The key observation is that even though $v$ might have negative entries, for $\theta$ large enough, the external field still has many large positive entries. 
    As we will see, this will force the Gibbs measure to have higher mass on vectors with a large portion of positive spins. 
    
    Let $h=\theta\one+v$ denote the external field and $\cS \coloneqq \{j \in [m]: h_j \leq \frac{\theta}{2}\}$ the set of indices with small (or negative) external field. Triangular inequality gives $\|v\|^2\geq |\cS|\left(\frac{\theta}{2}\right)^2$, so by taking $\theta > 4\frac{C_2}{\sqrt{\delta}}$, we have $|\cS|\leq \frac{\delta m}{4}$.
    
    For each $x \in \{\pm 1\}^m$ define the set of negative spins as $\cN(x) \coloneqq \{j \in [m]: x_j = -1\}.$
   
Since $|\cS|\leq \frac{\delta m}{4}$ we have
\begin{equation}\label{eq:unionboundmanyquantifiers}
\PP\left(\abs{\cN(X)} \geq \frac{\delta m}{2}\right) \leq \PP\left(\abs{\cN(X) \cap \cS^c}\geq \frac{\delta m}4 \right) = \sum_{\substack{\cM\subseteq\cS^c \\ |\cM|\geq\frac{\delta m}4}} \PP\left(\cN(X) \cap \cS^c = \cM \right).
\end{equation}

    Now, it suffices to upper bound $\PP(\cX_\cM)$ for each $\abs{\cM} \geq \frac{\delta m}{4}$, where \[\cX_\cM \coloneqq \{x \in \{\pm 1\}^m:  \cN(x) \cap \cS^c = \cM\}.\]

    In order to prove low probability of each set $\cX_\cM$, we construct, for each $\cM$, an injective map $z_\cM:\cX_\cM\to\{\pm1\}^m$ given by

    \begin{align*}
        z_{\cM}(x)_{j} \coloneqq \begin{cases}
            -x_j, &\quad \text{ if } j \in \cM \\
            x_j, &\quad \text{ if } j \notin \cM, 
        \end{cases}
    \end{align*}
    and upper bound $\frac{\nu_{J,h}(x)}{\nu_{J,h}(z_{\cM}(x))}$ for each $x\in\cX_\cM$.

Note that, since $x_j=-1$ for $j\in\cM$,  $z_{\cM}(x)-x = 2\mathbbm{1}_{\cM}$ and so $$x^\top J x = \left( z_{\cM}(x)- 2\mathbbm{1}_{\cM} \right)^\top J\left( z_{\cM}(x)- 2\mathbbm{1}_{\cM} \right) = z_{\cM}(x)^\top Jz_{\cM}(x) + 4 \mathbbm{1}_{\cM}^\top J\mathbbm{1}_{\cM} - 4\mathbbm{1}_{\cM}^\top Jz_{\cM}(x)$$ Thus, for $x\in\cX_\cM$, we have
    \begin{align*}
       \frac{\nu_{J,h}(x)}{\nu_{J,h}(z_{\cM}(x))} &= \exp\left(\frac12 x^\top Jx + h^\top x - \frac12 z_{\cM}(x)^\top Jz_{\cM}(x) - h^\top z_{\cM}(x)\right) \\ 
       & = \exp\left(2\big( \mathbbm{1}_{\cM}^\top J\mathbbm{1}_{\cM} - \mathbbm{1}_{\cM}^\top Jz_{\cM}(x) - h^\top \mathbbm{1}_{\cM}\big)\right)
       .
    \end{align*}
    
    We have $h^\top\mathbbm{1}_{\cM}\geq  \theta \abs{\cM}/2$, therefore, as long as $\theta$ is sufficiently large, the negative contribution $-h^\top\mathbbm{1}_{\cM}$ dominates:
    \begin{align*}
        \Big\lvert \mathbbm{1}_{\cM}^\top J\mathbbm{1}_{\cM} - \mathbbm{1}_{\cM}^\top Jz_{\cM}(x) \Big\rvert &\leq \left(|\cM| + \sqrt{|\cM|m} \right) \|J\| 
        \leq  \abs{\cM} \brac{1+ \brac{\frac{m}{\abs{\cM}}}^\frac{1}{2}} \|J\| \\ & \leq \abs{\cM} \brac{1+ \sqrt{\frac{4}{\delta}}} \|J\| \leq \frac{\theta \abs{\cM}}{4}.
    \end{align*}
    The last inequality follows from the assumption $\theta \geq 4 \brac{1+ \sqrt{\frac{4}{\delta}}}\|J\|$.
    Since $x \mapsto z_\cM(x)$ defines an injective map on the set $\cX_\cM$, 
    \begin{align*}
        \PP\brac{\cN(X) \cap \cS^c  = \cM} = \frac{\sum_{x \in \cX_{\cM}} \nu_{J,h}(x)}{\sum_{x \in \{\pm 1\}^m} \nu_{J,h}(x)} \leq \frac{\sum_{x \in \cX_{\cM}} \nu_{J,h}(x)}{\sum_{x \in \cX_\cM} \nu_{J,h}(z(x))} \leq e^{-\frac{\theta \abs{\cM}}{2}}.
    \end{align*}
    Plugging our upper bound in \eqref{eq:unionboundmanyquantifiers} gives
    \begin{align*}
        \PP\brac{\abs{\cN(X)} \geq \frac{\delta m}{2} } 
        &\leq \sum_{k \geq \lfloor\frac{\delta m}{4}\rfloor} \binom{m}{k} e^{-\frac{\theta k}{2}} \leq \sum_{k \geq \lfloor\frac{\delta m}{4}\rfloor} e^{ k \log(\frac{em}{k})-\frac{\theta k}{2}} \leq \sum_{k \geq \lfloor\frac{\delta m}{4}\rfloor} e^{ k \brac{\log(\frac{5e}{\delta})-\frac{\theta }{2}}}
    \end{align*}
where we used $\binom{m}{k} \leq \exp\brac{k\log(\frac{em}{k})}$ and we assumed $m$ sufficiently large such that $\lfloor\frac{\delta m}{4}\rfloor \geq \frac{\delta m}{5}$. Therefore, if $\theta >4\log\brac{\frac{5e}{\delta}}$, it holds that $\PP\brac{\abs{\cN(X)} \geq \frac{\delta m}{2} } \leq m e^{-\frac{\theta \delta m}{20}}$ which is smaller than or equal to $e^{-c_3 m}$ for sufficiently large $m$ and appropriately picked $c_3$.
\end{proof}

In order to use Lemma~\ref{lemma:highoverlapreplicas} we need a bound after conditioning on the replica disagreement at a single coordinate. As this reduces to a pinning of one coordinate, a new external field is added. Since this external field is small, adapting the proof of~\ref{lemma:manyquantifiers} to this setting is straightforward. 

\begin{corollary}\label{corollary:manyquantifiers}
Given the assumptions of Lemma~\ref{lemma:manyquantifiers}, there exists $\theta, c_3, m_0 \geq 0$ such that:

For every $i \in [m]$ and $\sigma_i \in \{\pm 1\}$, for every $m \geq m_0$, the random variable $X$ with quadratic Gibbs distribution (as in~\eqref{eq:gibbsinlemmamanyquantifiers}) on $\{\pm 1\}^m$ given by $\nu_{J,\theta \one_m +v}(x)$, satisfies that $X$, conditioned on $X_i = \sigma_i$, 
has, with probability at least $ 1-\exp(-c_3m)$, more than $(1-\frac{\delta}2)m$ positive spins.
\end{corollary}

\begin{proof}
    The conditional measure $\nu_{J,\theta \one_m +v}(x \mid x_i =\sigma_i)$ is an $(m-1)$-dimensional Gibbs measure on the remaining coordinates $x_{-i}$, with interaction matrix $J_{[m]\setminus\{i\},[m]\setminus\{i\}}$ and external field $h_j = \theta +v_j +J_{ij}\sigma_i$.

    As $\|J_{[m]\setminus\{i\},[m]\setminus\{i\}}\| \leq \|J\| \leq C_1$ and 
    \[\|v_{-i} +J_{i,-i}\sigma_i\| \leq \|v\| + \|J\|\leq \brac{2 C_2+C_1}\sqrt{m-1}.\]
    
    We can use Lemma~\ref{lemma:manyquantifiers} with the same $C_1$, with $C_2'=2C_2+C_1$, $c_3' =c_3/2$ and $\delta'$ such that $\left(1-\frac{\delta'}2\right)(m-1)\geq \left(1-\frac{\delta}2\right)m$. 
    Note that $\theta$ can simply be taken to be three times the $\theta$ in Lemma~\ref{lemma:manyquantifiers}.
\end{proof}

\section{Proof of the Main Results}\label{sec:mainproofs}
\subsection{Stochastic Localization with Pinnings}
The stochastic localization framework of~\cite{chen2022localizationschemesframeworkproving} allows us to obtain lower bound on the spectral gap of Glauber Dynamics from upper bounds on the correlation matrices of all pinned measures up to a certain pinning size. We formalize these guarantees below. The spectral independence framework from~\cite{anari2021spectral} is almost identical, with replacing correlation matrices by pairwise influence matrices and the same kind of worst-case analysis for pinned measures is required.

\begin{lemma}\label{lemma:spectralindconditionsformixing}
    Given a probability distribution $\nu$ on $\{\pm1\}^{n}$ and $0\leq k\leq n$ we say a measure $\nu'$ on $\{\pm1\}^{n-k}$ is a $k-$conditional submeasure of $\nu$ if it arises from $\nu$ after conditioning $k$ coordinates taking a set of values. We call the set of such measures $\mathcal{M}_k(\nu)$.

    Let $\kappa_k$ be defined as
    \[
\kappa_k \coloneqq \sup_{\nu'\in \mathcal{M}_k(\nu)} \left\| \Cor(\nu')\right\|.
    \]

Let $C_4, C_5
\geq 0$ be universal constants (not depending on $n$) with $C_5$ being an integer. For $\nu$ a distribution on the hypercube with $\min_{x \in \{\pm1 \}^n} \nu(x) \geq e^{-C_4 n}$, if there exists an integer $T_6$ satisfying $T_6>2\lceil C_5 \rceil$, for which:
\begin{itemize}
\item $\kappa_k \leq C_5 $ for all $k\leq n-T_6$
\item For $k=n-T_6+1$, Glauber mixes uniformly in polynomial time on any $\nu'\in\cM_k(\nu)$, 
\end{itemize}
then lazy Glauber dynamics mixes in polynomial time for $\nu$.
\end{lemma}

\begin{proof}
We show a polynomial lower bound on the spectral gap.
The proof follows from Section 3 and 4 in~\cite{chen2022localizationschemesframeworkproving}. We summarize it here for completeness.
Denote the transition kernel of Glauber dynamics with respect to a distribution $\nu$ on $\{\pm 1\}^n$ by $P_\nu$ and the corresponding Dirichlet form for functions $\psi: \{\pm 1\}^n \to \RR$ by $\cE_\nu(\psi,\psi)$.
    By Proposition 4.3 in~\cite{chen2022localizationschemesframeworkproving}, the Dirichlet form of Glauber Dynamics on random measures $\nu_t$, where $\nu_t$ is defined through a stochastic localization scheme initiated at $\nu$, is a supermartingale. Therefore, the spectral gap of Glauber dynamics with respect to the measure $\nu$ can be lower bounded by
    \begin{align*}
        1-\lambda_2(P_\nu) &= \min_{\substack{\psi: \{\pm 1\}^n \to \RR,\\ \Var_\nu(\psi) \neq 0}}\frac{\cE_{\nu}(\psi,\psi)}{\Var_\nu(\psi)}\geq \inf_{\substack{\psi: \{\pm 1\}^n \to \RR,\\ \Var_\nu(\psi) \neq 0}} \EE\bracc{\frac{\cE_{\nu_t}(\psi,\psi)}{\Var_{\nu}(\psi)}} \\
        &= \inf_{\substack{\psi: \{\pm 1\}^n \to \RR,\\ \Var_\nu(\psi) \neq 0}} \EE\bracc{\frac{\cE_{\nu_t}(\psi,\psi)}{\Var_{\nu_t}(\psi)}\frac{\Var_{\nu_t}(\psi)}{\Var_{\nu}(\psi)}} 
        \geq \inf_{\substack{\psi: \{\pm 1\}^n \to \RR,\\ \Var_\nu(\psi) \neq 0}} \EE\bracc{\brac{1-\lambda_2(P_{\nu_t})}\frac{\Var_{\nu_t}(\psi)}{\Var_{\nu}(\psi)}}.
    \end{align*}
    Therefore, if $\nu_t$ is defined by a stochastic localization process such that for all possible realisations it holds that $1-\lambda_2(P_{\nu_t}) \geq \eps>0$, the spectral gap of $P_\nu$ is bounded by $\eps \frac{\EE\bracc{\Var_{\nu_t}(\psi)}}{\Var_{\nu}(\psi)}$ (cf. Theorem 4.1 in~\cite{chen2022localizationschemesframeworkproving}). Here, $\eps$ may depend on $n$.
    Choose $\nu_t$ to be the stochastic localization scheme with pinnings and $t =  n-T_6  +1$. The random localized measure $\nu_t$ is in $\cM_t(\nu)$. By assumption, $1-\lambda_2(P_{\nu_t})\geq n^{-C_6}$ for a constant $C_6\geq 0$ since polynomial mixing time implies a lower bound on the spectral gap (cf.~\cite{levinperesmarkovchains}, Theorem 12.5). It suffices to bound
    \[
        \EE\bracc{\frac{\Var_{\nu_t}(\psi)}{\Var_\nu(\psi)}} = \EE\bracc{\prod_{k=0}^{t-1} \frac{\Var_{\nu_{k+1}}(\psi)}{\Var_{\nu_k}(\psi)}}
    \]
    uniformly over all functions $\psi: \{\pm 1\}^n \to \RR$.
    By the tower property, the expectation of this product can be bounded, if each factor $\EE\bracc{\frac{\Var_{\nu_{k+1}}(\psi)}{\Var_{\nu_k}(\psi)} \Bigm| \nu_k}$ can be bounded uniformly over the functions $\psi$ and in the realizations of the random measures $\nu_k$ (cf. Proposition 3.2 in~\cite{chen2022localizationschemesframeworkproving}).
    By an application of a Cauchy-Schwarz inequality and the explicit recursion $\nu_{k+1}\mid \nu_k$, it holds that
    \[
    \EE\bracc{\frac{\Var_{\nu_{k+1}}(\psi)}{\Var_{\nu_k}(\psi)} \Bigm| \nu_k} \geq 1- \frac{\|\Cor(\nu_k)\|_\op}{n-k} \geq 1-\frac{\kappa_k}{n-k}.
    \]
    For a proof of this statement we refer the reader to~\cite{chen2022localizationschemesframeworkproving}, Section 3.
    By combining the previous bounds, we conclude
    \[1-\lambda_2(P_\nu) \geq n^{-C_6}\prod_{k=0}^{t-1} \brac{1- \frac{\kappa_k}{n-k}}.\]
    As $\kappa_k \leq  C_5 $ for every $0 < k \leq t-1$, it holds that 
    \[1-\lambda_2(P_\nu) \geq n^{-C_6} \brac{\frac{(n- C_5 )! (n-t)!}{n! (n-t-C_5 )!}} \geq n^{-C_6} \brac{\frac{(n- C_5 )! ( T_6  -1)!}{n! ( T_6 - C_5  -1)!}} \geq n^{-(C_6+ C_5 )},\]
    where the expression $( T_6 - C_5 -1)!$ is well defined by the assumption $ 2   C_5  \leq  T_6$. A polynomial lower bound on the spectral gap implies a polynomial upper bound on the mixing time of the lazy version of the Markov chain (cf.~\cite{levinperesmarkovchains}, Theorem 12.4) when $\min_{x \in \{\pm1 \}^n}\nu(x) \geq e^{-C_4 n}$. 
\end{proof}

\subsection{Proof of Theorem~\ref{thm:smallsubmatricesimplymixing}}
Lemma~\ref{lemma:spectralindconditionsformixing} requires a bound on the correlation matrix of all pinned measures when many spins remain unpinned.
We show that every such measure is of the form~\eqref{eq:gibbsinlemmamanyquantifiers}.

\begin{proposition}\label{proposition:allpinningsgiveboundedexternalfields}
Let $C\geq 0$ and $\delta\in(0,1)$. Let $J\in\RR^{n\times n}$ symmetric such that $\|J\|\leq C$, and $h\in\RR^n$. Consider the quadratic Gibbs measure
       \[
    \nu(x) \propto \exp\left(\frac{1}2x^\top J x + h^\top x\right).
        \]
Let $\cP\subset [n]$ be a set of pinning coordinates of size $|\cP|\leq (1-\delta)n$ and $\tau_\cP\in\{\pm1\}^{\cP}$ be the pinned values. Let $\cS=\cP^c$ and $m=|\cS|$. The probability measure of $x_{\cS}\Bigm|x_{\cP}=\tau_\cP$ is a quadratic Gibbs measure on $\{\pm1\}^{\cS}$ of the form 
       \[
    \nu(x) \propto \exp\left(\frac{1}2x_{\cS}^\top J_{\cS,\cS} x_{\cS} + h_{\cS}^\top x_{\cS}+v^
    \top x_{\cS}\right),
        \]
with $\|v\|\leq C\sqrt{\frac{1-\delta}{\delta}}\sqrt{m}$.
\end{proposition}

\begin{proof}
A straightforward computation shows that the vector $v$ is given by $v=J_{\cS,\cP}\tau_{\cP}$ and so $\|v\|\leq \|J_{\cS,\cP}\|\, \|\tau_{\cP}\| \leq C \sqrt{|\cP|} \leq C \sqrt{\frac{1-\delta}{\delta}}\sqrt{m}. $
\end{proof}

We are now ready to prove the Theorem~\ref{thm:smallsubmatricesimplymixing}: this is done by verifying the assumption of Lemma~\ref{lemma:spectralindconditionsformixing}. Firstly, if less than $\delta n$ spins remain unpinned, the small-submatrix condition places the conditional measure in a high-temperature regime. Secondly, if at least $\delta n$ spins remain unpinned, the external field forces spin alignment (even under a worst-case pinning, as per Proposition~\ref{proposition:allpinningsgiveboundedexternalfields}).
Lemma~\ref{lemma:highoverlapreplicas} then gives a uniform correlation bound.

\begin{proof}[Proof of Theorem~\ref{thm:smallsubmatricesimplymixing}]
    First, verify the lower bound on \[\min_{x \in \{\pm 1\}^n} \nu(x)  \geq \frac{\exp\brac{-n \brac{\frac{1}{2} \|J\| + \abs{\theta}}}}{2^n\exp\brac{ n \brac{\frac{1}{2} \|J\| + \abs{\theta}}}} \geq e^{-C_4 n}\] for $C_4 = \log(2)+\|J\|+2 \theta$.
    
    In the setting of Lemma~\ref{lemma:spectralindconditionsformixing}, set $T_6 = \lfloor \delta n \rfloor$. We study the correlation matrices of all pinned measures $\nu_\tau$ where $\tau$ is a pinning of the spins at indices $\cP_\tau \subseteq [n]$.
    First consider $k = \abs{\cP_\tau} > n-T_6 = (1-\delta) n$. Since $|\cP_\tau^c|\leq \delta n$, mixing is guaranteed for any such submeasure $\nu' \in \cM_k(\nu_{J,\theta \one_n})$ by the spectral condition $\|J_{\cP_\tau^c,\cP_\tau^c}\|_\op < \frac{1}{2}$ (cf.~\cite{eldan2022spectral},~\cite{chen2022localizationschemesframeworkproving}).

    It remains to bound the correlation matrices by a constant $C_5$ for all pinned measures with $0 \leq \abs{\cP_\tau} \leq n(1-\delta)$. We apply Lemma~\ref{lemma:manyquantifiers} to show that the assumptions of Lemma~\ref{lemma:highoverlapreplicas} are satisfied which yields a bound on the desired correlation matrices. This proof will imply that $T_6 > 2C_5$ so indeed all assumptions from Lemma~\ref{lemma:spectralindconditionsformixing} hold.
    By Proposition~\ref{proposition:allpinningsgiveboundedexternalfields}, a pinning $\tau$ of size $\abs{\cP_\tau} \leq n(1-\delta)$ induces an additional external field with $C_2 = C_1 \sqrt{\frac{1-\delta}{\delta}}$.

    To apply Lemma~\ref{corollary:manyquantifiers}, we choose $\theta$ such that~\eqref{eq:bound_theta} is satisfied. In particular, let
    \begin{align}\label{eq:finalchoiceoftheta}
        \theta > 12 \max\left\{ \|J\|_\op\brac{1+\frac{2}{\sqrt{\delta}} }, \|J\|_\op\brac{\frac{2 \sqrt{1-\delta}+\sqrt{\delta}}{\delta} }, \log\brac{\frac{5e}{\delta}}\right\}.
    \end{align}
    Let $\nu_\tau$ be the measure pinned on $\abs{\cP_\tau} \leq (1-\delta)n$ spins. Let $m \coloneqq n-\abs{\cP_\tau}$. 
    An application of Corollary~\ref{corollary:manyquantifiers} gives:
    With probability at least $ 1-\exp(-c_3m)$, $X \sim \nu_{\tau}$, conditioned on $X_i = \sigma_i$ for any $i \in \cP_\tau^c$, has at least $(1-\frac{\delta}{2})m$ positive spins. 
    A bound on the overlap follows by the fact that for $\rho(X,X') \leq (1-\delta)m $ it needs to hold that either $X$ or $X'$ have at least $\frac{\delta}{2}m$ negative  spins. Therefore, we can bound the overlap probability by
    \begin{align}\label{eq:overlapbound}
        &\PP_{X,X' \sim \nu_{\tau}}\Big(\rho(X,X') \leq (1-\delta)m \Bigm| X_i \neq X_i'\Big) \nonumber
        =  \PP_{X,X' \sim \nu_{\tau}}\Big(\rho(X,X') \leq (1-\delta)m \Bigm| X_i = 1, X_i'=-1\Big) \nonumber\\
        &\leq \PP_{X \sim \nu_{\tau}}\Big(\sum_{i =1}^m X_i \leq (1-\delta)m \Bigm| X_i=1\Big) + \PP_{X' \sim \nu_{\tau}}\Big(\sum_{i =1}^m X_i' \leq (1-\delta)m \Bigm| X_i'=-1\Big)\nonumber\\
        &\leq 2 e^{-c_3m}
    \end{align}
    where the final bound follows from Corollary~\ref{corollary:manyquantifiers}.

    The covariance matrices $\Cov(\nu_{2J_{\cS,\cS}})$ for $\abs{\cS} \leq \delta n$ can be bounded by the decomposition from~\cite{bb19} (cf.~\cite{chen2022localizationschemesframeworkproving}, Lemma 5.2) which yields
    \begin{align}\label{eq:covboundsmallsets}
        \|\Cov(\nu_{2J_{\cS,\cS}})\| \leq \frac{1}{1-2\|2J_{\cS,\cS}\|_\op} \leq 5
    \end{align}
    by the assumption on $\delta$.
    
    It remains to combine~\eqref{eq:overlapbound} and~\eqref{eq:covboundsmallsets} with Lemma~\ref{lemma:highoverlapreplicas}: For all $k=\abs{\cP_\tau} \leq (1-\delta)n$, the correlation matrices are bounded by
    \[\kappa_k \leq  5+ 2(n-k) e^{-c_3(n-k)}\]
    which can be bounded by a universal constant $C_5$, which can be assumed to be an integer, for $n$ sufficiently large.
    Therefore, an application of Lemma~\ref{lemma:spectralindconditionsformixing} gives fast mixing of Glauber dynamics when $\theta$ satisfies~\eqref{eq:finalchoiceoftheta}.
\end{proof}
\begin{remark}
    In general, no effort has been made to optimize constants. In any case, we point out that the dependency $\|J\|/\delta$ can be improved to $\|J\|/\sqrt{\delta}$ if one chooses $\delta_m=\delta \frac{n}{m}$ and $C_2 = C_1\sqrt{\frac{n-m}{m}}$ depending on $m$ throughout all proofs. 
\end{remark}

\subsection{Proof of Corollary~\ref{cor:appliedtoSK}}
In order to apply Theorem~\ref{thm:smallsubmatricesimplymixing} to the SK model, one simply needs to verify the conditions of Theorem~\ref{thm:smallsubmatricesimplymixing} with high probability on the GOE matrix $W$. This is achieved by a routine application of Gaussian concentration combined with a union bound.

\begin{proof}[Proof of Corollary~\ref{cor:appliedtoSK}]
    We want to show that for any $\beta$, there exists $\delta$ small enough (not depending on $n$) so that $\beta W$ satisfies the condition of Theorem~\ref{thm:smallsubmatricesimplymixing} with a spectrally bounded interaction matrix.

    It is enough to show the condition for subsets $\cS\subset[n]$ of size $\lfloor{\delta n\rfloor}$ as submatrices cannot have larger spectral norm. 
    
    Since $W_{\cS,\cS}$ is a GOE matrix, the fact that $\EE\lambda_{\max}(\sqrt{n}W_{\cS,\cS})=-\EE\lambda_{\min}(\sqrt{n}W_{\cS,\cS})\leq 2\sqrt{|\cS|}$, combined with Gaussian concentration, gives, for all $t\geq 0$, (see, e.g.,~\cite{MDSbook})
    \[
\PP \left( \| \beta W_{\cS,\cS} \| \geq 2\beta\sqrt{\frac{|\cS|}{n}} + \frac{t\beta}{\sqrt{n}} \right) 
= \PP \left( \| \sqrt{n} W_{\cS,\cS} \| \geq 2\sqrt{|\cS|} + t \right) \leq 2 \exp\left(-t^2/4\right).
    \]
    
    For $S = [n]$, this gives $\left\| \beta W\right\| \leq 3 \beta $ with probability at least $1-\exp(-n/4)$.
    
    For $|\cS|=\lfloor\delta n\rfloor$, taking $\delta\leq\frac{1}{400\beta^2}$, and taking $t=\frac{\sqrt{n}}{10\beta}$ we have
    \[
\PP \left( \| \beta W_{\cS,\cS} \| \geq \frac15 \right) \leq 2 \exp\left(-\frac{n}{400\beta^2}\right).
    \]
    By the union bound, noting that $\delta<\frac12$,
    \[
\PP \left( \max_{\substack{\cS\subseteq [n] \\ |\cS| = \lfloor \delta n\rfloor}} \| \beta W_{\cS,\cS} \| \geq \frac15 \right) \leq 2 {n \choose  \delta n} \exp\left(-\frac{n}{400\beta^2}\right)\leq 2 \exp\left(\delta n\log\left(e/\delta\right)-\frac{n}{400\beta^2}\right),
    \]
    where we used ${n \choose k}\leq (\frac{en}{k})^k$. Taking $\delta$ small enough so that $\delta\log(e/\delta)\leq \frac1{401\beta^2}$ establishes the bound.
\end{proof}

\section{Further Applications and Open Questions}\label{sec:otherapplicationsandquestions}

There are many high overlap distributions of interest in which the tools developed in this paper can potentially be used. Any quadratic Gibbs model with suitably small interaction principal submatrices can fit the conditions of Theorem~\ref{thm:smallsubmatricesimplymixing}. In particular, the proof of Corollary~\ref{cor:appliedtoSK} can be easily adapted for variants of the SK model where the interaction matrix has another distribution (as long as submatrices can be suitably controlled). 

Furthermore, our results can be extended to non-constant external fields $h \in \RR^n$ as long as sufficiently many entries of $h$ is large enough in absolute value. This setting tends to appear in statistical estimation problems with side information. As a model problem, we describe below the $\ZZ_2$ Synchronization, a problem closely related to the problem of Community Detection in the Stochastic Block Model.

\subsection{$\ZZ_2$-Synchronization with side information}

The $\ZZ_2$ synchronization problem (\cite{singer2011angular,abbe2014decoding,cucuringu2015synchronization}) is a statistical estimation problem where one aims to estimate a signal $x^\natural\in\{\pm1\}^n$, with prior uniform over the hypercube, from noisy measurements which here we assume to be Gaussian
\begin{equation}\label{eq:Z2Synch:quadraticmeasurement}
Y = \frac1n\lambda\, x^\natural \left(x^\natural\right)^\top + W, 
\end{equation}
where $\lambda>0$ corresponds to the Signal-to-Noise-Ratio (SNR) and $W$ is a GOE matrix. A routine calculation shows that the posterior of the signal given the measurements is given by
\begin{equation}\label{eq:Z2posterior}
\Pr\left( x \bigm| Y \right) \propto \exp\left( 
\frac{\lambda}2x^\top Y x\right)
=
\exp\left( 
\frac{\lambda}2x^\top W x + \frac{\lambda^2}{2n}\left(x^\top x^\natural\right)^2 \right).
\end{equation}

Several sampling algorithms have been studied for this problem. Algorithmic stochastic localization in Wasserstein distance for $\lambda$ sufficiently large, but constant~\cite{montanari2023posterior}.
For a slightly modified version of the posterior, at higher temperature, it has been established that Glauber dynamics enters weak recovery regions after polynomial time \cite{liu2024locally} (several guarantees of recovery are also known for many algorithms when $\lambda >1$). 
It remains an open question whether Glauber dynamics mixes at a large SNR.

Let us consider a variant of this problem with side information\footnote{Various forms of side information models are studied in  Community Detection in the Stochastic Block Model, a notable example being the contextual SBM~\cite{Deshpandeetal:contextualSBM}, the simple model we describe here is a special case of the models in~\cite{BarbierReeves2020MultiviewSpikedMatrix}.} where, besides~\eqref{eq:Z2Synch:quadraticmeasurement}, one also has access to spin-level noisy measurements and an SNR $\eta\geq 0$,
\begin{equation}\label{eq:Z2Synch:linearmeasurements}
y = \eta \, x^\natural + \xi, 
\end{equation}
where $\xi\sim\cN(0,1)$, independent from the GOE matrix $W$.

The posterior is then given by
\[
\Pr\left( x \bigm| Y,y \right) \propto 
\exp\left( 
\frac{\lambda}2x^\top W x + \frac{\lambda^2}{2n}\left(x^\top x^\natural\right)^2 + \eta\,x^\top y \right)
= 
\exp\left( 
\frac{\lambda}2x^\top W x + \frac{\lambda^2}{2n}\left(x^\top x^\natural\right)^2 +  \eta^2\,x^\top x^\natural +  \eta\,x^\top \xi \right).
\]
In order to understand the spectral gap of Glauber Dynamics for this Gibbs measure we can, without loss of generality, take $x^\natural = \one$. In this case we have to consider the quadratic Gibbs measure
\begin{equation}\label{eq:QuadraticGibbsMesaurefromZ2SynchwithSideInfo}
    \nu(x) \propto \exp\left( \frac12 \, x^\top\left( \lambda W + \frac{\lambda^2}{n}\one\one^\top \right) x +  \eta^2\,x^\top \one +  \eta\,x^\top \xi \right),
\end{equation}
in the hypercube, for $W$ a GOE matrix and $\xi\sim\cN(0,\mathrm{I})$.

The following mixing guarantee can then be established with our tools.

\begin{corollary}\label{cor:Z2withSideInformation}
    For any SNR $\lambda>0$ and $n$ large enough, there exists $\eta\geq 0$, which may depend on $\lambda$ but not on $n$, such that Glauber Dynamics for the posterior of $\ZZ_2$ Synchronization with side information~\eqref{eq:QuadraticGibbsMesaurefromZ2SynchwithSideInfo} (where $W\sim\mathrm{GOE}$ and $\xi\sim\mathcal{N}(0,\mathrm{I})$), mixes, in total variation distance, in time polynomial in $n$ with high probability.
\end{corollary}

While we skip the details of this proof here, it closely follows that of Corollary~\ref{cor:appliedtoSK}. One does need to replace $\nu_{J,\theta \one_n}$ in Theorem~\ref{thm:smallsubmatricesimplymixing} by $\nu_{J,\theta \one_n + \eta \xi}$.
To adapt the proof, apply Lemma~\ref{lemma:manyquantifiers} and Corollary~\ref{corollary:manyquantifiers} for any pinned measure of~\eqref{eq:QuadraticGibbsMesaurefromZ2SynchwithSideInfo} with $C_1 \leq \|\lambda W + \frac{\lambda^2}{n}\one_n \one_n^\top \|$  and $C_2 \leq C_1 \sqrt{\frac{1-\delta}{\delta}} + \eta \sqrt{\frac{1}{\delta n}} \|\xi\|$. With high probability, it holds that $C_1 \leq 2\lambda + \lambda^2 +\epsilon$ for a small $\epsilon >0$. If $\lambda <\frac{1}{10}$, the spectral condition yields a direct proof, while for all other $\lambda$, it suffices to take $\delta$ such that $\delta \log(e/\delta) \leq \frac{1}{400 \lambda^2}$ by an analogous argument to the one from Corollary~\ref{cor:appliedtoSK}. 
As the required $\theta$ from~\eqref{eq:bound_theta} grows linearly in $C_2$, there exists $\eta$ such that $\theta = \eta^2$ and $\theta > 12 \max\left\{\frac{C_2}{\sqrt{\delta}}, \brac{1+\frac{2}{\sqrt{\delta}}}C_1,\log\left(\frac{5e}{\delta}\right)\right\}$, with high probability over $W$ and $\xi$.

\section*{Acknowledgments and use of AI}
AE would like to thank Frederic Koehler for sharing with him, sometime in the fall of 2022, the conjecture that for any temperature, the SK model should enjoy fast coldstarted Glauber mixing at strong enough magnetization. ASB and AR would like to thank Yuansi Chen for numerous insightful discussions on the topic in this paper. AR would like to thank Roland Bauerschmidt for an instructive discussion on this topic, and for pointing out the connections with~\cite{percus1975correlation}.

We used modern AI tools in this work, primarily ChatGPT Pro 5.5. 
In fact, what are (in our view) the two key ingredients in our argument were found by ChatGPT: One is the realization that it is not too lossy to attempt to control correlations matrices for all possible pinnings (even though some pinnings do flip the magnetization of remaining spins, they do so of only very few); the other is the use of the high overlap replica trick to control correlations by correlations of small submeasures (which we learned from Roland dates back to the 70s).

The paper was written entirely by the authors. AI was used to scan the manuscript
for typos and inconsistencies in style and notation. All errors are ours.

AE was supported by the National Science Foundation grant DMS-2450867.
% ------------------------------------------------------------
% Bibliography
% ------------------------------------------------------------

\bibliographystyle{alpha}
\bibliography{references}

\end{document}